\documentclass[10pt]{amsart}
\usepackage{amssymb}
\newtheorem{lemme}{Lemma}

\newtheorem{theorem}{Theorem}
\newtheorem{remark}{Remark}

\begin{document}

\title [Huygens type inequalities for Bessel and modified Bessel Functions ]{ Extension of Huygens type inequalities for Bessel and modified Bessel Functions \\}%

\author[ Khaled Mehrez]{KHALED MEHREZ }
\address{Khaled Mehrez. D\'epartement de Math\'ematiques ISSAT Kasserine, Universit\'e de Kairouan, Tunisia.}
 \email{k.mehrez@yahoo.fr}
\begin{abstract}
 In this paper, new sharpened Huygens type inequalities involving Bessel and  modified Bessel functions
are established. 
\end{abstract}
\maketitle
{\it keywords:} The Bessel functions, The modified Bessel functions, Huygens type inequalities.  \\

\section{\textbf{Introduction}}

This inequality 
\begin{equation}\label{0001}
2\frac{\sin x}{x}+\frac{\tan x}{x}>3
\end{equation}
which holds for all $x\in(0,\pi/2)$  is known in literature as Huygens's inequality \cite{H}. The hyperbolic counterpart of (\ref{0001}) was established in \cite{N} as follows:
\begin{equation}\label{0002}
2\frac{\sinh x}{x}+\frac{\tanh x}{x}>3,\;x>0. 
\end{equation}
The inequalities (\ref{0001}) and (\ref{0002}) were respectively refined in \cite{H} as 
\begin{equation}
2\frac{\sin x}{x}+\frac{\tan x}{x}>2\frac{x}{\sin x}+\frac{x}{\tan x}>3
\end{equation}
for $0<x<\frac{\pi}{2}$ and 
\begin{equation}
2\frac{\sinh x}{x}+\frac{\tanh x}{x}>2\frac{x}{\sinh x}+\frac{x}{\tanh x}>3,\;x\neq0.
\end{equation}
Recently, in \cite {zhu7}, Zhu give some new inequalities of the Huygens type for circular functions,
hyperbolic functions, and the reciprocals of circular and hyperbolic functions, as follows:\\
\textbf{Theorem A}
The following inequalities 
\begin{equation}\label{sw02}
(1-p)\frac{x}{\sin x}+p\frac{x}{\tan x}>1>(1-q)\frac{x}{\sin x}+q\frac{x}{\tan x}
\end{equation}
holds for all $x\in(0, \pi/2)$ if and only if $p\leq1/3$ and $q\geq1-2/\pi.$\\
\noindent\textbf{Theorem B} The following inequalities 
\begin{equation}\label{sw01}
(1-p)\frac{\sin x}{x}+p\frac{\tan x}{x}>1>(1-q)\frac{\sin x}{x}+q\frac{\tan x}{x}
\end{equation}
holds for all $x\in(0, \pi/2)$ if and only if $p\geq1/3$ and $q\leq 0.$\\
\textbf{Theorem C}
 The following inequalities 
\begin{equation}\label{sw03}
(1-p)\frac{\sinh x}{x}+p\frac{\tanh x}{x}>1>(1-q)\frac{\sinh x}{x}+q\frac{\tanh x}{x}
\end{equation}
holds for all $x\in(0, \infty)$ if and only if $p\leq1/3$ and $q\geq 1.$\\
\textbf{Theorem D} The following inequalities 
\begin{equation}\label{sw04}
(1-p)\frac{x}{\sinh x}+p\frac{x}{\tanh x}>1>(1-q)\frac{x}{\sinh x}+q\frac{x}{\tanh x}
\end{equation}
holds for all $x\in(0, \infty)$ if and only if $p\geq1/3$ and $q\leq 0$.\\
 
In this paper, we first give a generalizations of inequalities (\ref{sw02}) and (\ref{sw01}) to Bessel functions of the first kind and present an conjecture, which may be of interest for further research. Second, we extend and sharpen inequalities  (\ref{sw03}) and (\ref{sw04}) for the modified Bessel functions of the first kind.

\section{\textbf{Lemmas}}
We begin this section with the following useful lemmas which are needed to completes
the proof of the main theorems.
\begin{lemme}\label{l1}\cite{kh, and, pin} Let $f,g:[a,b]\longrightarrow\mathbb{R}$ be two continuous functions which are differentiable on $(a,b)$. Further, let $g^{'}\neq0$ on $(a,b).$ If $\frac{f^{\prime}}{g^{\prime}}$ is increasing (or decreasing) on $(a,b)$, then the functions $\frac{f(x)-f(a)}{g(x)-g(a)}$ and $\frac{f(x)-f(b)}{g(x)-g(b)}$ are also increasing (or decreasing) on $(a,b).$
\end{lemme}
\begin{lemme}\label{l2}\cite{ponn} Let $a_n$ and $b_n\;(n=0,1,2,...)$ be real numbers,  and let the power series $A(x)=\sum_{n=0}^{\infty}a_{n}x^{n}$ and $B(x)=\sum_{n=0}^{\infty}b_{n}x^{n}$ be convergent for $|x|<R.$ If $b_n>0$ for $n=0,1,..,$ and if $\frac{a_n}{b_n}$ is strictly increasing(or decreasing) for $n=0,1,2...,$ then the function $\frac{A(x)}{B(x)}$ is strictly increasing (or decreasing) on $(0,R).$
\end{lemme}
\section{\textbf{Extensions of Huygens type inequalities to Bessel functions}}

In this section, our aim is to extend the inequalities (\ref{sw02}) and (\ref{sw01}) to Bessel functions of the first
kind. For this suppose that $\nu>-1$ and consider the function $\mathcal{J}_{\nu}:\mathbb{R}\longrightarrow(-\infty, 1],$ defined by
\begin{equation*}
\mathcal{J}_{\nu}(x)=2^\nu\Gamma(\nu+1)x^{-\nu}J_{\nu}(x)=\sum_{n\geq}\frac{\left(\frac{-1}{4}\right)^n}{(\nu+1)_n n!}x^{2n},
\end{equation*}
where $\Gamma$ is the gamma function, $(\nu+1)_n=\Gamma(\nu+n+1)/\Gamma(\nu+1)$ for each $n\geq0$, is the well-known Pochhammer (or Appell) symbol, and $J_\nu$ defined by
$$J_\nu(x)=\sum_{n\geq0}\frac{(-1)^n(x/2)^{\nu+2n}}{n!\Gamma(\nu+n+1)},$$
stands for the Bessel function of the first kind of order $\nu$. It is worth mentioning that in
particular the function $J_\nu$  reduces to some elementary functions, like sine and cosine.
More precisely, in particular we have:
\begin{equation}\label{ee}
\mathcal{J}_{-1/2}(x)=\sqrt{\pi/2}.x^{1/2}J_{-1/2}(x)=\cos x,
\end{equation}
\begin{equation}\label{ee1}
\mathcal{J}_{1/2}(x)=\sqrt{\pi/2}.x^{-1/2}J_{1/2}(x)=\frac{\sin x}{x},
\end{equation}
respectively, which can verified easily by using the series representation of the function $
J_\nu$  and of the cosine and sine functions, respectively. Taking into account the relations (\ref{ee}) and (\ref{ee1}) as we mentioned above, the inequalities (\ref{sw02}) and (\ref{sw01}) can be rewritten in terms of $\mathcal{J}_{-1/2}(x)$ and $\mathcal{J}_{1/2}(x).$ For example, using (\ref{ee}) and  (\ref{ee1}) the inequalities (\ref{sw02}) and (\ref{sw01}) can be rewritten as
\begin{equation}
(1-p)\frac{1}{\mathcal{J}_{1/2}(x)}+p\frac{\mathcal{J}_{-1/2}(x)}{\mathcal{J}_{1/2}(x)}>1>(1-q)\frac{1}{\mathcal{J}_{1/2}(x)}+q\frac{\mathcal{J}_{-1/2}(x)}{\mathcal{J}_{1/2}(x)}
\end{equation}
and
\begin{equation}
(1-p)\mathcal{J}_{1/2}(x)+p\frac{\mathcal{J}_{1/2}(x)}{\mathcal{J}_{-1/2}(x)}>1>(1-q)\mathcal{J}_{1/2}(x)+q\frac{\mathcal{J}_{1/2}(x)}{\mathcal{J}_{-1/2}(x)}
\end{equation}
and thus it is natural to ask what is the general form of the inequalities (\ref{sw02}) and (\ref{sw01}) for
arbitrary $\nu$.\\ 

Our first main result is an extension of inequalities (\ref{sw02}) to Bessel functions of the first kind $J_\nu.$

\begin{theorem} Let $\nu>-1$ and  let $j_{\nu,1}$ the first positive zero of the Bessel function $J_\nu$ of the first kind. Then the Huygens types inequalities 
\begin{equation}\label{051}
(1-p)\frac{1}{\mathcal{J}_{\nu+1}(x)}+p\frac{\mathcal{J}_{\nu}(x)}{\mathcal{J}_{\nu+1}(x)}>1>(1-q)\frac{1}{\mathcal{J}_{\nu+1}(x)}+q\frac{\mathcal{J}_{\nu}(x)}{\mathcal{J}_{\nu+1}(x)},
\end{equation}
holds for all $x\in(0,j_{\nu,1})$ if and only if $p\leq\frac{\nu+1}{\nu+2}$ and $q\geq1-\mathcal{J}_\nu(j_{\nu,1}).$
\end{theorem}
\begin{proof} We define the function $F_{\nu}(x)$ on $(0,j_{\nu,1})$ by 
$$F_{\nu}(x)=\frac{\frac{1}{\mathcal{J}_{\nu+1}(x)}-1}{\frac{1}{\mathcal{J}_{\nu+1}(x)}-\frac{\mathcal{J}_{\nu}(x)}{\mathcal{J}_{\nu+1}(x)}}=\frac{1-\mathcal{J}_{\nu+1}(x)}{1-\mathcal{J}_{\nu}(x)}=\frac{h_{\nu,1}(x)}{h_{\nu,2}(x)},
$$ 
where $f_{\nu,1}(x)=1-\mathcal{J}_{\nu+1}(x)$ and $f_{\nu,2}(x)=1-\mathcal{J}_{\nu}(x).$ Now, by again using the differentiation formula
\begin{equation}\label{555}
\mathcal{J}_{\nu}^\prime(x)=-\frac{x}{2(\nu+1)}\mathcal{J}_{\nu+1}(x) 
\end{equation}
and the infinite product representation [\cite{wat}, p. 498]
\begin{equation}\label{5500}
\mathcal{J}_{\nu}(x)=\prod_{n\geq1}\left(1-\frac{x^{2}}{j_{\nu,n}^{2}}\right)
\end{equation}
we obtain that
\begin{equation*}
\begin{split}
\frac{f_{\nu,1}^\prime(x)}{f_{\nu,2}^\prime(x)}&=\frac{(\nu+1)\mathcal{J}_{\nu+2}(x)}{(\nu+2)\mathcal{J}_{\nu+1}(x)}\\
&=4(\nu+1)\sum_{n\geq1}\frac{1}{j_{\nu+1,n}^2-x^2}
\end{split}
\end{equation*}
So
$$\left(\frac{f_{\nu,1}^{\prime}(x)}{f_{\nu,2}^{\prime}(x)}\right)^{\prime}=8(\nu+1)\sum_{n\geq1}\frac{x}{(j_{\nu+1,n}^{2}-x^{2})^{2}}.$$
From this, we deduce that the function $\frac{f_{\nu,1}^{\prime}(x)}{f_{\nu,2}^{\prime}(x)}$ is increasing on $(0,j_{\nu,1}).$ Thus, the function $F_\nu(x)$ is also increasing on  $(0,j_{\nu,1})$ by Lemma \ref{l2}.

In view of $\lim_{x\longrightarrow0^+}F_{\nu}(x)=\frac{\nu+1}{\nu+2}$ and $\lim_{x\longrightarrow j_{\nu,1}}F_{\nu}(x)=1-\mathcal{J}_{\nu+1}(j_{\nu,1}).$ With this the proof is complete. 
\end{proof}    
\begin{theorem} Let $-1<\nu\leq0$ and  let $j_{\nu,1}$ the first positive zero of the Bessel function $J_\nu$ of the first kind. Then the Huygens type inequalities
\begin{equation}\label{050}
(1-p)\mathcal{J}_{\nu+1}(x)+p\frac{\mathcal{J}_{\nu+1}(x)}{\mathcal{J}_{\nu}(x)}>1>(1-q)\mathcal{J}_{\nu+1}(x)+q\frac{\mathcal{J}_{\nu+1}(x)}{\mathcal{J}_{\nu}(x)}
\end{equation}
holds for all $(x\in(0, j_{\nu,1}),$ if and only if, $p\geq\frac{\nu+1}{\nu+2}$ and $q\leq0.$
\end{theorem}
\begin{proof} Let $\nu>-1$, consider the function 
$$G_\nu(x)=\frac{1-\mathcal{J}_{\nu+1}(x)}{\frac{\mathcal{J}_{\nu+1}(x)}{\mathcal{J}_{\nu}(x)}-\mathcal{J}_{\nu+1}(x)},\;\;0<x<j_{\nu,1}.$$
For $0<x<j_{\nu,1}$, let 
$$g_{\nu,1}(x)=1-\mathcal{J}_{\nu+1}(x)\;\textrm{and}\;g_{\nu,2}(x)=\frac{\mathcal{J}_{\nu+1}(x)}{\mathcal{J}_{\nu}(x)}-\mathcal{J}_{\nu+1}(x).$$
From the differentiation formula (\ref{555}), we get
$$\frac{g_{\nu,1}^{\prime}(x)}{g_{\nu,2}^{\prime}(x)}=\frac{1}{1+\frac{1}{\mathcal{J}_{\nu}(x)}\Bigg(\frac{\nu+2}{\nu+1}.\frac{\mathcal{J}_{\nu+1}^2(x)}{\mathcal{J}_{\nu}(x)\mathcal{J}_{\nu+2}(x)}-1\Bigg)}=\frac{1}{1+\frac{L_\nu(x)}{\mathcal{J}_{\nu}(x)}}$$
where 
$$L_\nu(x)=\frac{\nu+2}{\nu+1}.\frac{\mathcal{J}_{\nu+1}^2(x)}{\mathcal{J}_{\nu}(x)\mathcal{J}_{\nu+2}(x)}-1.$$
On other hand, using the Tur\'an type inequality [\cite{arb1}, eq. 2.9]
\begin{equation}\label{5555}
\mathcal{J}_{\nu+1}^2(x)-\mathcal{J}_{\nu}(x)\mathcal{J}_{\nu+2}(x)>0.
\end{equation}
where $\nu>-1$ and $x\in(-j_{\nu,1},j_{\nu,1}),$ we obtain that the function $L_\nu(x)$ is positive on $(0,j_{\nu,1}).$\\
Elementary  calculations reveal that
\begin{equation}
L_{\nu}^{\prime}(x)=
\frac{(\nu+2)x\mathcal{J}_{\nu+1}(x)}{(\nu+1)\mathcal{J}_{\nu}^{2}(x)\mathcal{J}_{\nu+2}(x)}\left[\frac{\mathcal{J}_{\nu+1}^{2}(x)}{2(\nu+1)}-\frac{\mathcal{J}_{\nu}(x)\mathcal{J}_{\nu+2}(x)}{\nu+2}\right]+\frac{(\nu+2)x\mathcal{J}_{\nu+1}^{2}(x)\mathcal{J}_{\nu+2}(x)}{2(\nu+3)(\nu+1)\mathcal{J}_{\nu}(x)\mathcal{J}_{\nu+2}^{2}(x)}\end{equation}
Using the fact that $\mathcal{J}_{\nu+1}(x)\geq\mathcal{J}_{\nu}(x)>0$ for all $x\in(0,j_{\nu,1})$ and the Tur\'an type inequality (\ref{5555}), we get
\begin{equation}
L_{\nu}^{\prime}(x)\geq \frac{-\nu x\mathcal{J}_{\nu+1}^{3}(x)}{2(\nu+1)^2\mathcal{J}_{\nu}^{2}(x)\mathcal{J}_{\nu+2}^{2}(x)}+\frac{(\nu+1)x\mathcal{J}_{\nu+1}^{2}(x)\mathcal{J}_{\nu+3}(x)}{2(\nu+3)(\nu+1)\mathcal{J}_{\nu}(x)\mathcal{J}_{\nu+2}^{2}(x)}
\end{equation}
Thus, we conclude that the function $L_\nu(x)$ is increasing on $(0,j_{\nu,1})$ for all $-1<\nu\leq0.$ Since the function $x\longmapsto\mathcal{J}_{\nu}(x)$ is decreasing (\cite{arb1}, Theorem 3) on $(0,j_{\nu,1}),$ we gave that the function $\frac{L_\nu(x)}{\mathcal{J}_{\nu}(x)}$ is increasing too on $(0,j_{\nu,1}),$ as a product of two positives increasing functions. Thus, the function  $\frac{g_{\nu,1}^{\prime}(x)}{g_{\nu,2}^{\prime}(x)}$ is decreasing on $(0,j_{\nu,1}).$ Then, the function
$$G_\nu(x)=\frac{g_{\nu,1}(x)}{g_{\nu,2}(x)}=\frac{g_{\nu,1}(x)-g_{\nu,1}(0)}{g_{\nu,2}(x)-g_{\nu,2}(0)}.$$
is decreasing on $(0,j_{\nu,1})$, by Lemma \ref{l2},\\
At the same time, we can write the function $G_\nu(x)$ in the following form
$$G_\nu(x)=\frac{\mathcal{J}_{\nu}(x)}{\mathcal{J}_{\nu+1}(x)}.F_\nu(x).$$
So
$$\lim_{x\longrightarrow0^{+}}G_\nu(x)=F_\nu(0)=\frac{\nu+1}{\nu+2}\;\;\textrm{and}\;\lim_{x\longrightarrow j_{\nu,1}}G_\nu(x)=0,$$
and with this the proof of inequalities (\ref{050}) is done. 
\end{proof}
\begin{remark}
Since $j_{-1/2,1}=\frac{\pi}{2}$ we find that the inequalities (\ref{051}) and (\ref{050}) is the generalization of inequalities (\ref{sw02}) and (\ref{sw01}).
\end{remark}
\noindent\textbf{Conjecture.} The function 
$$x\longmapsto G_\nu(x)=\frac{1-\mathcal{J}_{\nu+1}(x)}{\frac{\mathcal{J}_{\nu+1}(x)}{\mathcal{J}_{\nu}(x)}-\mathcal{J}_{\nu+1}(x)},$$
is decreasing on $(0,j_{\nu,1})$ and $\nu>-1.$ If our present conjecture were correct, then this would lead to a extended the inequalities (\ref{050}).\\

\section{\textbf{Extensions of the Huygens type inequalities to modified Bessel functions}}

In this section, we present a generalization of inequalities (\ref{sw03}) and (\ref{sw04}). For $\nu>-1$ let us consider the function $\mathcal{I}_{\nu}:\mathbb{R}\longrightarrow[1,\infty),$ defined by
\begin{equation*}
\mathcal{I}_{\nu}(x)=2^\nu\Gamma(\nu+1)x^{-\nu}I_{\nu}(x)=\sum_{n\geq}\frac{\left(\frac{1}{4}\right)^n}{(\nu+1)_n n!}x^{2n},
\end{equation*}
where $I_\nu$ is the  modified Bessel function of the first kind defined by
$$I_\nu(x)=\sum_{n\geq0}\frac{(x/2)^{\nu+2n}}{n!\Gamma(\nu+n+1)},\; \textrm{for all}\; x\in\mathbb{R}.$$
It is worth mentioning that in particular we have
\begin{equation}
\mathcal{I}_{-1/2}(x)=\sqrt{\pi/2}.x^{1/2}I_{-1/2}(x)=\cosh x,
\end{equation}
\begin{equation}
\mathcal{J}_{1/2}(x)=\sqrt{\pi/2}.x^{1/2}I_{-1/2}(x)=\frac{\sinh x}{x}.
\end{equation}
Thus, the function $I_\nu$ is of special interest in this paper because inequalities (\ref{sw03}) and  (\ref{sw04}) is actually
equivalent to
\begin{equation}\label{eee1}
\left(1-p\right)\mathcal{I}_{1/2}(x)+p\frac{\mathcal{I}_{1/2}(x)}{\mathcal{I}_{-1/2}(x)}>1>\left(1-q\right)\mathcal{I}_{1/2}(x)+q\frac{\mathcal{I}_{1/2}(x)}{\mathcal{I}_{-1/2}(x)},
\end{equation}
for all $x\in(0,\infty)$ if and only if $p\leq\frac{-1/2+1}{-1/2+2}=1/3$ and $q\geq1$, and
\begin{equation}\label{eee2}
\left(1-p\right)\frac{1}{\mathcal{I}_{-1/2+1}(x)}+p\frac{\mathcal{I}_{-1/2}(x)}{\mathcal{I}_{-1/2+1}(x)}>1>\left(1-q\right)\frac{1}{\mathcal{I}_{-1/2+1}(x)}+q\frac{\mathcal{I}_{-1/2}(x)}{\mathcal{I}_{-1/2+1}(x)},
\end{equation}
for all $x\in(0,\infty)$ if and only if $p\geq\frac{-1/2+1}{-1/2+2}=1/3$ and $q\leq0.$\\

So in view of inequalities (\ref{eee1}) and (\ref{eee2})  it is natural to ask: what is the analogue of this inequalities
for modified Bessel functions of the first kind? In order to answer this question we prove the following results.

\begin{theorem}\label{t3}
Let $\nu>-1$, the following inequalities 
\begin{equation}\label{mm0}
\left(1-p\right)\mathcal{I}_{\nu+1}(x)+p\frac{\mathcal{I}_{\nu+1}(x)}{\mathcal{I}_{\nu}(x)}>1>\left(1-q\right)\mathcal{I}_{\nu+1}(x)+q\frac{\mathcal{I}_{\nu+1}(x)}{\mathcal{I}_{\nu}(x)},
\end{equation}
holds for all $x\in(0,\infty)$ if and only if $p\leq\frac{\nu+1}{\nu+2}$ and $q\geq1$.
\end{theorem}
\begin{proof}
Let $\nu>-1$, we define the function $H_\nu$ on $(0,\infty)$ by 
\begin{equation*}
H_\nu(x)=\frac{\mathcal{I}_{\nu+1}(x)-1}{\mathcal{I}_{\nu+1}(x)-\frac{\mathcal{I}_{\nu+1}(x)}{\mathcal{I}_{\nu}(x)}}=\frac{\mathcal{I}_{\nu+1}(x)\mathcal{I}_{\nu}(x)-\mathcal{I}_{\nu}(x)}{\mathcal{I}_{\nu+1}(x)\mathcal{I}_{\nu}(x)-\mathcal{I}_{\nu+1}(x)}=\frac{h_{\nu,1}(x)}{h_{\nu,2}(x)},
\end{equation*}
where $h_{\nu,1}(x)=\mathcal{I}_{\nu+1}(x)\mathcal{I}_{\nu}(x)-\mathcal{I}_{\nu}(x)$ and $h_{\nu,2}(x)=\mathcal{I}_{\nu+1}(x)\mathcal{I}_{\nu}(x)-\mathcal{I}_{\nu+1}(x).$
By using the differentiation formula [\cite{wat}, p. 79]
\begin{equation}\label{mm}
\mathcal{I}_{\nu}^{\prime}(x)=\frac{x}{2(\nu+1)}\mathcal{I}_{\nu+1}(x)
\end{equation}
can easily show that
\begin{equation}
h_{\nu,1}^{\prime}(x)=\frac{x}{2(\nu+1)}\mathcal{I}_{\nu+1}^{2}(x)+\frac{x}{2(\nu+2)}\mathcal{I}_{\nu}(x)\mathcal{I}_{\nu+2}(x)-\frac{x}{2(\nu+1)}\mathcal{I}_{\nu+1}(x),
\end{equation}
and
\begin{equation}
h_{\nu,2}^{\prime}(x)=\frac{x}{2(\nu+1)}\mathcal{I}_{\nu+1}^{2}(x)+\frac{x}{2(\nu+2)}\mathcal{I}_{\nu}(x)\mathcal{I}_{\nu+2}(x)-\frac{x}{2(\nu+2)}\mathcal{I}_{\nu+2}(x)
\end{equation}
Using the Cauchy product 
\begin{equation}
I_{\mu}(x)I_{\nu}(x)=\sum_{n\geq0}\frac{\Gamma(\nu+\mu+2n+1)x^{\nu+\mu+2n}}{2^{\mu+\nu+2n}\Gamma(n+1)\Gamma(\nu+\mu+n+1)\Gamma(\mu+n+1)\Gamma(\nu+n+1)}
\end{equation}
we obtain 
\begin{equation}
h_{\nu,1}^{\prime}(x)=\sum_{n\geq0}A_n(\nu)x^{2n}
\end{equation}
and
\begin{equation}
h_{\nu,2}^{\prime}(x)=\sum_{n\geq0}B_n(\nu)x^{2n}
\end{equation}
where
\begin{equation}
A_{n}(\nu)=\frac{\Gamma(\nu+1)\Big(\Gamma(\nu+2)\Gamma(2\nu+2n+4)-\Gamma(2\nu+n+3)\Gamma(\nu+n+3)\Big)}{2^{2n+1}\Gamma(n+1)\Gamma(\nu+n+2)\Gamma(\nu+n+3)\Gamma(2\nu+n+3)}
\end{equation}
and
\begin{equation}
B_{n}(\nu)=\frac{\Gamma(\nu+2)\Big(\Gamma(\nu+1)\Gamma(2\nu+2n+4)-\Gamma(2\nu+n+3)\Gamma(\nu+n+2)\Big)}{2^{2n+1}\Gamma(n+1)\Gamma(\nu+n+2)\Gamma(\nu+n+3)\Gamma(2\nu+n+3)}
\end{equation}
Now, we define the sequence $C_n=\frac{A_n}{B_n}$ for $n=0,1,...,$ thus
$$C_n(\nu)=\frac{\Gamma(\nu+1)\Big(\Gamma(\nu+2)\Gamma(2\nu+2n+4)-\Gamma(2\nu+n+3)\Gamma(\nu+n+3)\Bigg)}{\Gamma(\nu+2)\Big(\Gamma(\nu+1)\Gamma(2\nu+2n+4)-\Gamma(2\nu+n+3)\Gamma(\nu+n+2)\Big)}
$$
So, for $\nu>-1$ and $n=0,1,...,$ we get 
\begin{equation}
\begin{split}
\frac{C_{n+1}(\nu)}{C_{n}(\nu)}&=\frac{\left[\Gamma(\nu+2)\Gamma(2\nu+2n+6)-\Gamma(2\nu+n+4)\Gamma(\nu+n+4)\right]}{\left[\Gamma(\nu+1)\Gamma(2\nu+2n+6)-\Gamma(2\nu+n+4)\Gamma(\nu+n+3)\right]}\\
&\times \frac{\left[\Gamma(\nu+1)\Gamma(2\nu+2n+4)-\Gamma(2\nu+n+3)\Gamma(\nu+n+2)\right]}{\left[\Gamma(\nu+2)\Gamma(2\nu+2n+4)-\Gamma(2\nu+n+3)\Gamma(\nu+n+3)\right]}\\
&=\frac{K_n^1(\nu)}{K_n^2(\nu)}
\end{split}
\end{equation}
where
\begin{equation*}
\begin{split}
K_n^1(\nu)&=\left[\Gamma(\nu+2)\Gamma(2\nu+2n+6)-\Gamma(2\nu+n+4)\Gamma(\nu+n+4)\right]\left[\Gamma(\nu+1)\Gamma(2\nu+2n+4)-\Gamma(2\nu+n+3)\Gamma(\nu+n+2)\right]\\
&=\underbrace{\Gamma(\nu+1)\Gamma(\nu+2)\Gamma(2\nu+2n+6)\Gamma(2\nu+2n+4)}_{A_1}-\underbrace{\Gamma(\nu+2)\Gamma(2\nu+2n+6)\Gamma(2\nu+n+3)\Gamma(\nu+n+2)}_{B_1}\\
&-\underbrace{\Gamma(\nu+1)\Gamma(2\nu+n+4)\Gamma(\nu+n+4)\Gamma(2\nu+2n+4)}_{C_1}+\underbrace{\Gamma(2\nu+n+4)\Gamma(\nu+n+4)\Gamma(2\nu+n+3)\Gamma(\nu+n+2)}_{D_1}
\end{split}
\end{equation*}
and 
\begin{equation*}
\begin{split}
K_n^2(\nu)&=\left[\Gamma(\nu+1)\Gamma(2\nu+2n+6)-\Gamma(2\nu+n+4)\Gamma(\nu+n+3)\right]\left[\Gamma(\nu+2)\Gamma(2\nu+2n+4)-\Gamma(2\nu+n+3)\Gamma(\nu+n+3)\right]\\
&=\underbrace{\Gamma(\nu+1)\Gamma(\nu+2)\Gamma(2\nu+2n+6)\Gamma(2\nu+2n+4)}_{A_1}-\underbrace{\Gamma(\nu+1)\Gamma(2\nu+2n+6)\Gamma(2\nu+n+3)\Gamma(\nu+n+3)}_{B_2}\\
&-\underbrace{\Gamma(\nu+2)\Gamma(2\nu+n+4)\Gamma(\nu+n+3)\Gamma(2\nu+2n+4)}_{C_2}+\underbrace{\Gamma(2\nu+n+4)\Gamma^2(\nu+n+3)\Gamma(2\nu+n+3)}_{D_2}
\end{split}
\end{equation*}
Thus $$K_n^1(\nu)-K_n^2(\nu)=(B_2-B_1)+(C_2-C_1)+(D_1-D_2)$$
A simple calculation we obtain 
$$B_2-B_1=(n+1)\Gamma(\nu+1)\Gamma(\nu+n+2)\Gamma(2\nu+n+3)\Gamma(2\nu+2n+6)$$
and
$$C_2-C_1=-(n+2)\Gamma(\nu+1)\Gamma(2\nu+2n+4)\Gamma(2\nu+n+4)\Gamma(\nu+n+3)$$
and
$$D_1-D_2=\Gamma(2\nu+n+4)\Gamma(2\nu+n+3)\Gamma(\nu+n+3)\Gamma(\nu+n+2)\geq0.$$
Then simple computations lead to
\begin{equation*}
\begin{split}
B_2-B_1+(C_2-C_1)&=\Gamma(\nu+1)\Gamma(2\nu+2n+4)\Gamma(2\nu+n+3)\Gamma(\nu+n+2)\Bigg((n+1)(2\nu+2n+5)(2\nu+2n+4)\\
&-(n+2)(2\nu+2n+3)(\nu+n+2)\Bigg)\\
&\geq P_n(\nu)(n+1)\Gamma(\nu+1)\Gamma(2\nu+2n+4)\Gamma(2\nu+n+3)\Gamma(\nu+n+2)
\end{split}
\end{equation*}
where 
\begin{equation*}
\begin{split}
P_n(\nu)&=2\nu^2+(4n+11)\nu+2n^2+11n+14\\
&=(\nu+n+2)(2\nu+2n+7)>0,
\end{split}
\end{equation*}
for all $\nu>-1$ and $n\in\mathbb{N}.$ Therefore, the sequence $(C_n)_n$ is increasing, we obtain that the function $\frac{h_{\nu,1}^{\prime}(x)}{h_{\nu,2}^{\prime}(x)}$ is increasing on $(0,\infty)$ too ( by Lemma \ref{l1}). Thus $H_\nu(x)=\frac{h_{\nu,1}(x)-h_{\nu,1}(0^+)}{h_{\nu,2}(x)-h_{\nu,2}(0^+)}$ is also increasing on $(0,\infty)$  by Lemma \ref{l2}. So, 
$$\lim_{x\longrightarrow0^{+}}H_{\nu}(x)=C_{0}(\nu)=\frac{\nu+1}{\nu+2},$$
and using the asymptotic formula [\cite{1}, p. 377]
\[
I_{\nu}(x)=\frac{e^{x}}{\sqrt{2\pi x}}\left[1-\frac{4\nu^{2}-1}{1!(8x)}+\frac{(4\nu^{2}-1)(4\nu^{2}-9)}{2!(8x)^{2}}-...\right]\]
which holds for large values of $x$ and for fixed $\nu > - 1$, we obtain
$$\lim_{x\longrightarrow\infty}H_{\nu}(x)=1.$$
So the proof of Theorem \ref{t3} is complete.
\end{proof}
\begin{theorem}\label{t2} Let $\nu>-1$, the following inequalities 
\begin{equation}\label{mm1}
\left(1-p\right)\frac{1}{\mathcal{I}_{\nu+1}(x)}+p\frac{\mathcal{I}_{\nu}(x)}{\mathcal{I}_{\nu+1}(x)}>1>\left(1-q\right)\frac{1}{\mathcal{I}_{\nu+1}(x)}+q\frac{\mathcal{I}_{\nu}(x)}{\mathcal{I}_{\nu+1}(x)},
\end{equation}
holds for all $x\in(0,\infty)$ if and only if $p\geq\frac{\nu+1}{\nu+2}$ and $q\leq0$.
\end{theorem}
\begin{proof}Let $\nu>-1$ and $x\in(0,\infty)$, we define the function $\Phi_{\nu}(x)$ by
\begin{equation}
\Phi_{\nu}(x)=\frac{\frac{1}{\mathcal{I}_{\nu+1}(x)}-1}{\frac{1}{\mathcal{I}_{\nu+1}(x)}-\frac{\mathcal{I}_{\nu}(x)}{\mathcal{I}_{\nu+1}(x)}}=\frac{1-\mathcal{I}_{\nu+1}(x)}{1-\mathcal{I}_{\nu}(x)}=\frac{\varphi_{\nu,1} (x)}{\varphi_{\nu,2} (x)},
\end{equation}
where $\varphi_{\nu,1}(x)=1-\mathcal{I}_{\nu+1}(x)$ and $\varphi_{\nu,2}(x)=1-\mathcal{I}_{\nu}(x).$ By again using the differentiation formula (\ref{mm}) we get 
\begin{equation}
\frac{\varphi_{\nu,1}^{\prime}(x)}{\varphi_{\nu,2}^{\prime}(x)}=\frac{\nu+1}{\nu+2}.\frac{\mathcal{I}_{\nu+2}(x)}{\mathcal{I}_{\nu+1}(x)}=\frac{\sum_{n=0}^{\infty}a_n(\nu)x^{2n}}{\sum_{n=0}^{\infty}b_n(\nu)x^{2n}},
\end{equation}
where $a_n(\nu)=\frac{\Gamma(\nu+2)}{2^{2n}\Gamma(n+1)\Gamma(\nu+n+3)}$ and $b_n(\nu)=\frac{\Gamma(\nu+1)}{2^{2n}\Gamma(n+1)\Gamma(\nu+n+2)}.$ Let $$c_n(\nu)=\frac{a_n(\nu)}{b_n(\nu)}=\frac{\nu+1}{\nu+n+2},\textrm{for}\;n=0,1,....$$
We conclude that $c_n(\nu)$ is decreasing for $n=0,1,...$ and $\frac{g_{1}^{\prime}(x)}{g_{2}^{\prime}(x)}$ is decreasing on $(0,\infty)$ by Lemma \ref{l2}. Thus
$$\Phi_{\nu}(x)=\frac{\varphi_{\nu,1} (x)}{\varphi_{\nu,2} (x)}=\frac{\varphi_{\nu,1} (x)-\varphi_{\nu,1} (0)}{\varphi_{\nu,2}(x)-\varphi_{\nu,2} (0)},$$
is decreasing on $(0,\infty)$ by Lemma \ref{l1}.
Furthermore,
\begin{equation*}
\lim_{x\longrightarrow0^{+}}\Phi_{\nu}(x)=c_{0}(\nu)=\frac{\nu+1}{\nu+2},
\end{equation*}
and
\begin{equation*}
\lim_{x\longrightarrow\infty}\Phi_{\nu}(x)=0.
\end{equation*}
Alternatively, inequality (\ref{mm1}) can be proved by using the Mittag-Leffler expansion  for the modified Bessel functions
of first kind, which becomes [\cite{erd}, Eq. 7.9.3]
\begin{equation}\label{44}
\frac{I_{\nu+1}(x)}{I_{\nu}(x)}=\sum_{n=1}^{\infty}\frac{2x}{j_{\nu,n}^{2}+x^{2}},
\end{equation}
where $0<j_{\nu,1}<j_{\nu,2}<...<j_{\nu,n}<...$, are the positive zeros of the Bessel function  $J_\nu$, we obtain that 
$$\frac{g_{1}^{\prime}(x)}{g_{2}^{\prime}(x)}=2(\nu+1)\frac{I_{\nu+2}(x)}{xI_{\nu+1}(x)}=4(\nu+1)\sum_{n=1}^{\infty}\frac{1}{j_{\nu,n}^{2}+x^{2}}.$$
\end{proof}
Clearly,
$$\left(\frac{g_{1}^{\prime}(x)}{g_{2}^{\prime}(x)}\right)^{\prime}=-8(\nu+1)\sum_{n=1}^{\infty}\frac{x}{(x^{2}+j_{\nu,n}^{2})^{2}},$$
for all $x>0$ and $\nu>-1$, which implies that $G_\nu(x)$ is decreasing for all $\nu>-1.$ On the other hand, using  the Rayleigh formula [\cite{wat}, p. 502 ]
\begin{equation}\label{55}
\sum_{n=1}^{\infty}\frac{1}{j_{\nu,n}^{2}}=\frac{1}{4(\nu+1)}.
\end{equation}
we get
\begin{equation*}
\lim_{x\longrightarrow0^{+}}G_{\nu}(x)=\frac{\nu+1}{\nu+2}.
\end{equation*}
So, the proof of Theorem \ref{t2} is complete.
\begin{remark} Since  $\mathcal{I}_{-\frac{1}{2}}(x)=\cosh x$ and $\mathcal{I}_{\frac{1}{2}}(x)=\frac{\sinh x}{x}$,  we find that the inequalities (\ref{mm0}) and (\ref{mm1}) is the generalization of inequalities (\ref{sw03}) and (\ref{sw04}).
\end{remark}

\end{document}